\documentclass[12pt]{article}
\usepackage{}

\usepackage{epsfig}
\usepackage{latexsym}
\usepackage{caption}
\usepackage{amsfonts}
\usepackage{amssymb}
\usepackage{mathrsfs}
\usepackage{amsmath}
\usepackage{enumerate}
\usepackage{graphics}
\usepackage{float}
\usepackage{pict2e}

\usepackage{color}

\makeatletter

\newcommand{\Rmnum}[1]{\expandafter\@slowromancap\romannumeral #1@}

\makeatother

\begin{document}

\newtheorem{theorem}{Theorem}
\newtheorem{observation}[theorem]{Observation}
\newtheorem{corollary}[theorem]{Corollary}
\newtheorem{algorithm}[theorem]{Algorithm}
\newtheorem{definition}[theorem]{Definition}
\newtheorem{guess}[theorem]{Conjecture}
\newtheorem{claim}[theorem]{Claim}
\newtheorem{problem}[theorem]{Problem}
\newtheorem{question}[theorem]{Question}
\newtheorem{lemma}[theorem]{Lemma}
\newtheorem{proposition}[theorem]{Proposition}
\newtheorem{fact}[theorem]{Fact}

\makeatletter
  \newcommand\figcaption{\def\@captype{figure}\caption}
  \newcommand\tabcaption{\def\@captype{table}\caption}
\makeatother

\newtheorem{acknowledgement}[theorem]{Acknowledgement}

\newtheorem{axiom}[theorem]{Axiom}
\newtheorem{case}[theorem]{Case}
\newtheorem{conclusion}[theorem]{Conclusion}
\newtheorem{condition}[theorem]{Condition}
\newtheorem{conjecture}[theorem]{Conjecture}
\newtheorem{criterion}[theorem]{Criterion}
\newtheorem{example}[theorem]{Example}
\newtheorem{exercise}[theorem]{Exercise}
\newtheorem{notation}{Notation}
\newtheorem{solution}[theorem]{Solution}
\newtheorem{summary}[theorem]{Summary}

\newenvironment{proof}{\noindent {\bf
Proof.}}{\rule{3mm}{3mm}\par\medskip}
\newcommand{\remark}{\medskip\par\noindent {\bf Remark.~~}}
\newcommand{\pp}{{\it p.}}
\newcommand{\de}{\em}
\newcommand{\mad}{\rm mad}
\newcommand{\qf}{Q({\cal F},s)}
\newcommand{\qff}{Q({\cal F}',s)}
\newcommand{\qfff}{Q({\cal F}'',s)}
\newcommand{\f}{{\cal F}}
\newcommand{\ff}{{\cal F}'}
\newcommand{\fff}{{\cal F}''}
\newcommand{\fs}{{\cal F},s}
\newcommand{\s}{\mathcal{S}}
\newcommand{\G}{\Gamma}
\newcommand{\g}{\gamma}
\newcommand{\wrt}{with respect to }
\newcommand {\nk}{ Nim$_{\rm{k}} $  }

\newcommand{\q}{\uppercase\expandafter{\romannumeral1}}
\newcommand{\qq}{\uppercase\expandafter{\romannumeral2}}
\newcommand{\qqq}{\uppercase\expandafter{\romannumeral3}}
\newcommand{\qqqq}{\uppercase\expandafter{\romannumeral4}}
\newcommand{\qqqqq}{\uppercase\expandafter{\romannumeral5}}
\newcommand{\qqqqqq}{\uppercase\expandafter{\romannumeral6}}

\newcommand{\qed}{\hfill\rule{0.5em}{0.809em}}

\newcommand{\var}{\vartriangle}

\title{{\large \bf Multiple list colouring of  planar graphs}}

\author{   Xuding Zhu\thanks{Department of Mathematics, Zhejiang Normal University,  China.  E-mail: xudingzhu@gmail.com. Grant Number: NSFC 11571319.},}

\date{}
\maketitle

\begin{abstract}
This paper proves that for each positive integer $m$, there is a planar graph $G$ which is not $(4m+\lfloor \frac{2m-1}{9}\rfloor,m)$-choosable. Then we pose some conjectures concerning multiple list colouring of planar graphs.  
\end{abstract}

\vspace{3mm}\textbf{Keywords:} fractional choice number,  multiple list colouring,  planar graph.

\section{Introduction}

A $b$-fold colouring of a graph $G$ is a mapping $\phi$ which assigns to each vertex $v$ of $G$ a set $\phi(v)$ of $b$ colours, so that adjacent vertices receive disjoint colour sets. An $(a,b)$-colouring of $G$ is a $b$-fold colouring $\phi$ of $G$ such that $\phi(v) \subseteq \{1,2,\ldots, a\}$ for each vertex $v$. The {\em fractional chromatic number} of $G$ is
 $$\chi_f(G)= \inf\{\frac ab: G \ \text{ is $(a,b)$-colourable } \}.$$
An $a$-list assignment of $G$ is a mapping $L$ which assigns to each vertex $v$ a set $L(v)$ of $a$ permissible colours. A $b$-fold $L$-colouring of $G$ is a $b$-fold colouring $\phi$ of $G$ such that $\phi(v) \subseteq L(v)$ for each vertex $v$. We say $G$ is $(a,b)$-choosable if for any $a$-list assignment $L$ of $G$, there is a $b$-fold $L$-colouring of $G$. The {\em fractional choice number} of $G$ is $$ch_f(G)= \inf\{\frac ab:  G \ \text{is $(a,b)$-choosable }\}.$$ It was proved by Alon, Tuza and Voigt \cite{ATV1995} that for any finite graph $G$, $\chi_f(G)=ch_f(G)$ and moreover the infimum in the definition of $ch_f(G)$ is attained and hence can be replaced by minimum. This implies that if $G$ is $(a,b)$-colourable, then for some integer $m$, $G$ is $(am,bm)$-choosable. As every planar graph is $4$-colourable (which is equivalent to $(4,1)$-colourable), we know for each planar graph $G$, there is an integer $m$ such that $G$ is $(4m,m)$-choosable. However, the integer $m$ depends on $G$. 
We prove in this paper that there is no integer $m$ so that every planar graph is $(4m,m)$-choosable.   

\begin{theorem}
	\label{main}
	For each positive integer $m$, there is a planar graph $G$ which is not $(4m+\lfloor \frac{2m-1}{9}\rfloor,m)$-choosable.
\end{theorem}

The $m=1$ case of Theorem \ref{main} is equivalent to say that there are non-$4$-choosable planar graphs, which was  proved by   Voigt \cite{Voigt1993}. A smaller non-$4$-choosable planar graph was constructed \cite{Mir1996}, a $3$-colourable non-$4$-choosable planar graph was constructed by Gutner   \cite{Gut1996,VW1997} and it was prove in \cite{Gut1996} that it is NP-complete to decide if a given planar graph is $4$-choosable.


\section{The proof of Theorem \ref{main}}

In this section, $m$ is a fixed positive integer. Let $k= \lfloor \frac{2m-1}{9}\rfloor$. We shall construct a planar graph 
$G$ which is not $(4m+k,m)$-choosable.

\begin{figure}[hh!]
	\label{fig2}
	\centering
	\includegraphics[width=140mm]{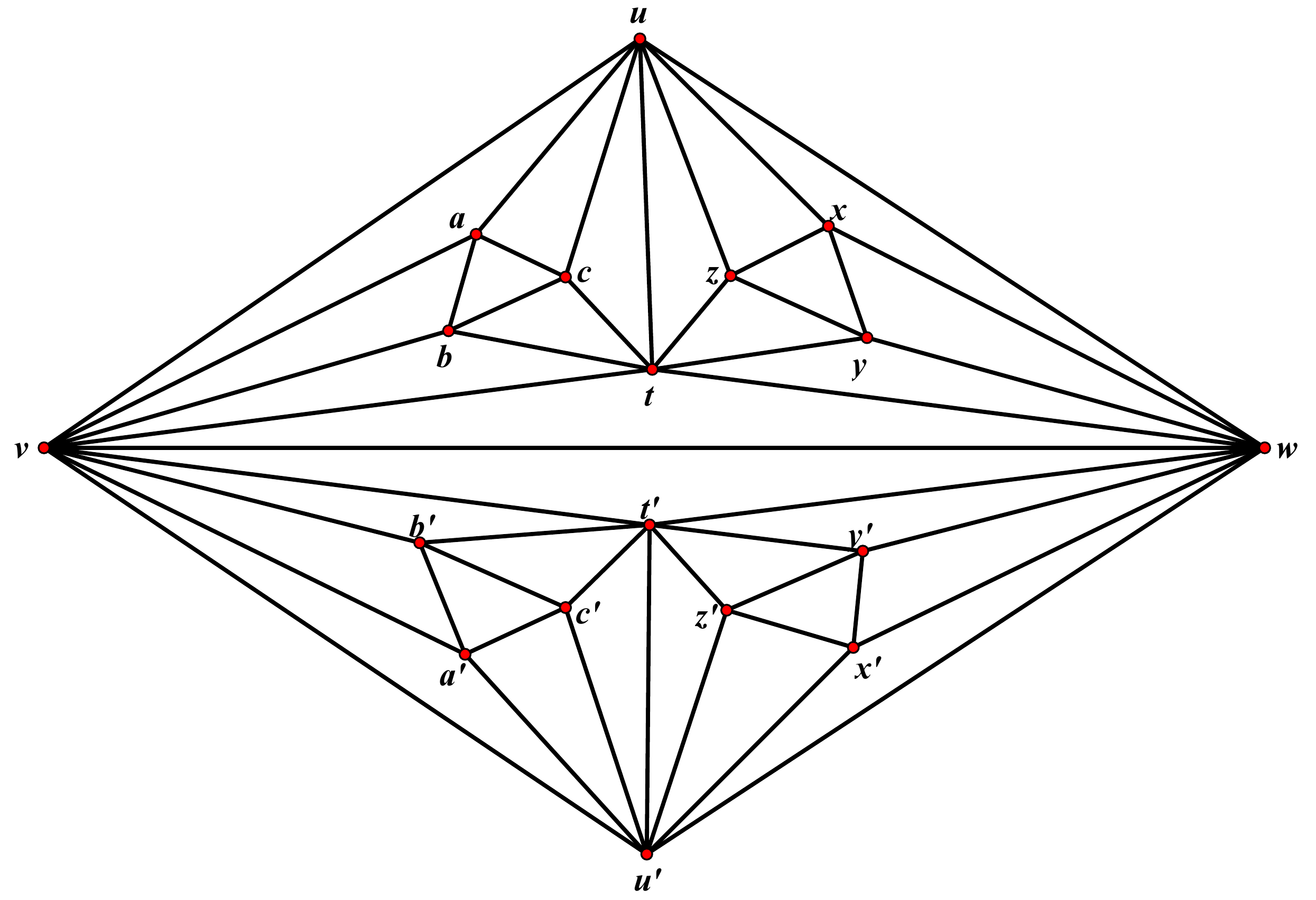}
	\label{fig1}
	\caption{The graph $G$}
\end{figure}

\begin{lemma}
	\label{lem1}
	Let $G$ be the graph as shown in Figure 1. There is a list assignment $L$
	of $G$ for which the following hold:
	\begin{enumerate}
		\item $|L(s)|=4m+k$ for each vertex $s$, except that  $|L(u)|=|L(u')|=m$.
		\item There is no $m$-fold $L$-colouring of $G$.
	\end{enumerate} 
\end{lemma}
\begin{proof}
	Let $A, B, C,D$ be pairwise disjoint sets of colours such that $|A|=|B|=m$ and $|C|=|D|=2m+k$. Let $X, X'$ be disjoint subsets of $C$ with $|X|=|X'|=m$. Define list assignment $L$ as follows:
	\begin{itemize}
		\item $L(u)=A$ and $L(u')=B$.
		\item $L(v)=L(w)=L(t)=L(t')=A \cup B \cup C$.
		\item $L(x)=L(a)=X \cup A \cup D$, and $L(x')=L(a')=X' \cup A \cup D$.
		\item $L(y)=L(b)=X \cup B \cup D$, and $L(y')=L(b')=X' \cup B \cup D$.
		\item $L(z)=L(c)= L(z')=L(c')=A \cup B \cup D$.
	\end{itemize}
	We shall show that there is no $m$-fold $L$-colouring of $G$.
	Assume to the contrary that $\phi$ is an $m$-fold $L$-colouring of $G$.
	Then $\phi(u)=A$ and $\phi(u')=B$ and $\phi(v), \phi(w)$ are two disjoint $m$-subsets of $C$.
	
	Let $S=G-(\phi(v) \cup \phi(w))$. We have $|S|=k$.
	As $X, X'$ are  disjoint subsets of $C$, we have
	$$|X \cap S|+|X' \cap S| \le |S|=k.$$
	By symmetry between the two triangles $(u,v,w)$ and $(u',v,w)$, we may assume that 
	$$|X \cap S| \le |X' \cap S|, \ \text{and hence} \ |X \cap S| \le k/2.$$
	(In case $|X' \cap S| \le \frac k2$, then we consider the subgraph contained in the triangle $(u',v,w)$).
	Thus
	$$|X \cap \phi(v)| +|X \cap \phi(w)|=|X \cap (\phi(v) \cup \phi(w))| \ge m- \frac k2.$$
	By symmetry between triangles $(u,v,t)$ and $(u,w,t)$, we may assume that 
	$$|X \cap \phi(v)| \ge |X \cap \phi(w)|, \ \text{and hence } \ |X \cap \phi(v)| \ge \frac m2 - \frac k4.$$
		(In case $|X \cap \phi(w)| \le \frac m2 - \frac k4$, then we consider the subgraph contained in the triangle $(u,w,t)$).
		
	Let $T=X-\phi(v)$. We have $$|T|=|X|-|X \cap \phi(v)| \le \frac m2 + \frac k4.$$
	Let $R=B - \phi(t)$. As $\phi(t)$ is disjoint from $\phi(u) \cup \phi(v) \phi(w)$,
	we know that $\phi(t) \subseteq B \cup S$. Hence
	$$|R|\le |S|= k.$$
By deleting the colours used by the neighbours of $a,b,c$, respectively, we have
	\begin{itemize}
		\item $\phi(a) \subseteq D \cup T$,
		\item $\phi(b) \subseteq D \cup R$,
		\item $\phi(c) \subseteq D \cup T \cup R$.
	\end{itemize}
	As $\phi(a),\phi(b),\phi(c)$ are pairwise disjoint, we have
	$$3m = |\phi(a) \cup \phi(b) \cup \phi(c)| \le |D|+|T|+|R| 
	\le (2m+k)+\left(\frac m2 + \frac k4\right) + k = \frac{5m}{2}+ \frac{9k}{4} < 3m,$$
	a contradiction.
\end{proof}
 
Let $p={4m+k \choose m, m,2m+k}$, and let $H$ be obtained from the disjoint union of $p$ copies of $G$ by identifying all the copies of $u$ into a single vertex (also named as $u$) and all the copies of $u'$ into a single vertex (also named as $u'$), and then add an edge connecting $u$ and $u'$. It is obvious that $H$ is a planar graph. 

Now we show that   $H$ is not $(4m+k,m)$-choosable.  Let $Z$ be a set of $4m+k$ colours.
Let $L(u)=L(u')=Z$. There are $p$ possible $m$-fold $L$-colourings of $u$ and $u'$. Each such a colouring $\phi$ corresponds to one copy of $G$. In that copy of $G$, define the list assignment as in the proof of Lemma \ref{lem1}, by replacing $A$ with $\phi(u)$ and $B$ with $\phi(u')$.  Now Lemma \ref{lem1} implies that no $m$-fold colouring of $u$ and $u'$ can be extended to an $m$-fold $L$-colouring of $H$. This completes the proof of Theorem \ref{main}.

\section{Some open problems}

Thomassen proved that every planar graph is $5$-choosable \cite{Tho1994}. The proof can be easily adopted to show that for any positive integer $m$, every planar graph is $(5m,m)$-choosable. Given a positive integer $m$, let $a(m)$ be the minimum integer such that every planar graph is $(a(m),m)$-choosable.
Combining Thomassen's result and Theorem \ref{main}, we have $$4m+\lfloor \frac{2m-1}{9}\rfloor+1 \le a(m) \le 5m.$$
 For $m=1$, the upper bound and the lower bound coincide. So $a(1)=5$. As $m$ becomes bigger, the gap between the upper and lower bounds increases.
 A natural question is what is the exact value of $a(m)$. We conjecture that the upper bound is not always tight.

\begin{conjecture}
	\label{conj1}
	There is a constant integer $m$ such that every planar graph is $(5m-1,m)$-choosable.
\end{conjecture}

It is proved recently in \cite{HZ2016} that every planar graph is   $1$-defective
$(9,2)$-paintable, which implies that every planar graph is $1$-defective $(9,2)$-choosable (i.e., if each vertex has $9$ permissible colours, then there is a
$2$-fold colouring of the vertices of $G$ with permissible colours so that each colour class induces a graph of maximum degree at most $1$.) 
The following conjecture, which is stronger than Conjecture \ref{conj1}, asserts that the $1$-defective can be replaced by $0$-defective.

\begin{conjecture}
	\label{conj2}
	Every planar graph is $(9,2)$-choosable.
\end{conjecture}

It follows from the Four Colour Theorem  that for any integer $m$, every planar graph is $(4m,m)$-colourable. Without using the Four Colour Theorem, it is proved very recently by Cranston and Rabern  \cite{CR2016}  that every planar graph is $(9,2)$-colourable (an earlier result in \cite{HRS1973} shows that every planar graph $G$ is $(5m-1,m)$-colourable with $m = |V(G)|+1$).

\bibliographystyle{unsrt}

\end{document}